\documentclass[12pt,twoside]{amsart}
\usepackage{amssymb,amsmath,amsthm, amscd, enumerate, mathrsfs}
\usepackage{graphicx, hhline}
\usepackage[all]{xy}
\usepackage[usenames]{color}
\usepackage{hyperref}
\hypersetup{colorlinks=true}

\title[Basepoint-free theorem of Reid--Fukuda type]
{Basepoint-free theorem of Reid--Fukuda type 
for quasi-log schemes}
\author{Osamu Fujino}
\date{2015/7/20, version 1.05}
\address{Department of Mathematics, Graduate School of Science,
Kyoto University, Kyoto 606-8502, Japan}
\email{fujino@math.kyoto-u.ac.jp}

\newcommand{\Bs}[0]{\operatorname{Bs}}
\newcommand{\Nklt}[0]{\operatorname{Nklt}}
\newcommand{\Nqklt}[0]{\operatorname{Nqklt}}
\newcommand{\Nqlc}[0]{\operatorname{Nqlc}}

\newcommand{\Supp}[0]{\operatorname{Supp}}
\newcommand{\Spec}[0]{\operatorname{Spec}}
\newcommand{\Exc}[0]{\operatorname{Exc}}
\newtheorem{thm}{Theorem}[section]
\newtheorem{lem}[thm]{Lemma}

\newtheorem*{claim}{Claim}

\theoremstyle{definition}
\newtheorem{ex}[thm]{Example}
\newtheorem{defn}[thm]{Definition}
\newtheorem{rem}[thm]{Remark}
\newtheorem*{ack}{Acknowledgments}
\newtheorem{notation}[thm]{Notation}
\newtheorem{say}[thm]{}
\newtheorem{step}{Step}

\makeatletter
    
    \@addtoreset{equation}{section}
\makeatother
\begin{document}
\bibliographystyle{amsalpha+}

\maketitle

\begin{abstract}
We introduce various new operations for quasi-log structures. 
Then we prove the basepoint-free theorem of Reid--Fukuda type 
for quasi-log schemes as an application. 
\end{abstract}
\tableofcontents

\section{Introduction}\label{f-sec1}

Let $(X, \Delta)$ be a log canonical 
pair and let $f:Y\to X$ be a resolution such that 
$K_Y+\Delta_Y=f^*(K_X+\Delta)$ and $\Supp \Delta_Y$ is a simple normal crossing 
divisor on $Y$. 
We put $S=\Delta_Y^{=1}$ and $\Delta_Y=S+B$. 
We consider the short exact sequence
$$
0\to \mathcal O_Y(-S+\lceil -B\rceil)\to \mathcal O_Y(\lceil -B\rceil)\to 
\mathcal O_S(\lceil -B\rceil)\to 0. 
$$ 
By the Kawamata--Viehweg vanishing theorem, we have 
$R^1f_*\mathcal O_Y(-S+\lceil 
-B\rceil)=0$. 
Therefore, we obtain 
$$
0\to \mathcal J(X, \Delta)\to \mathcal O_X\to f_*\mathcal O_S(\lceil -B\rceil)\to 0
$$ 
where $\mathcal J(X, \Delta)=f_*\mathcal O_Y(-S+\lceil -B\rceil)$ is the multiplier ideal sheaf of $(X, \Delta)$. 
Let $\Nklt (X, \Delta)$ be the non-klt locus of $(X, \Delta)$ with the 
reduced scheme structure. 
Then $f_*\mathcal O_S(\lceil -B\rceil)\simeq \mathcal O_{\Nklt(X, \Delta)}$. 
This data 
$$
f:(S, B|_S)\to \Nklt(X, \Delta)
$$ 
is a typical example of {\em{quasi-log schemes}}. 
In general, $\Nklt(X, \Delta)$ is reducible and is not equidimensional. 
Note that the data 
$$
f:(Y, \Delta_Y)\to X
$$ also defines a natural quasi-log structure on $X$ which is 
compatible with the original log canonical structure of $(X, \Delta)$. 
By the framework of quasi-log schemes, we can treat log canonical 
pairs and their non-klt loci on an equal footing. 

The following theorem is the main theorem of this paper. 
It was stated in \cite{ambro} without 
proof (see \cite[Theorem 7.2]{ambro} and Remark \ref{f-rem1.4}). 
For some related results, 
see \cite[10.4]{shokurov}, \cite{fukuda1}, \cite{fukuda2}, \cite{fukuda-lc}, 
\cite{fujino-reid-fukuda}, \cite[5.~Basepoint-free theorem of 
Reid--Fukuda type]{fujino-base}, and \cite[Theorem 
1.16]{fujino-slc}. 
Note that the comment by Professor Miles Reid in \cite[\S 10]{shokurov} 
is the origin of this type of basepoint-free theorems. 

\begin{thm}[Basepoint-free theorem of 
Reid--Fukuda type for quasi-log schemes]\label{f-thm1.1}
Let $[X, \omega]$ be a quasi-log scheme, 
let $\pi:X\to S$ be a projective morphism between schemes, 
and let $L$ be a $\pi$-nef Cartier divisor on $X$ such that 
$qL-\omega$ is nef and log big over $S$ with respect to $[X, \omega]$ for some positive 
real number $q$. Assume that 
$\mathcal O_{X_{-\infty}}(mL)$ is $\pi$-generated for every $m\gg 0$. 
Then $\mathcal O_X(mL)$ is $\pi$-generated 
for every $m\gg 0$. 
\end{thm}

In \cite[Theorem 4.1]{fujino-book}, the author proved Theorem \ref{f-thm1.1} 
with the extra assumption that $X_{-\infty}=\emptyset$. 
Note that the assumption $X_{-\infty}=\emptyset$ 
is harmless for applications to semi-log canonical pairs in \cite[Theorem 
1.16]{fujino-slc}.  
We also note that Ambro's original statement (see \cite[Theorem 7.2]{ambro}) 
only requires that $\pi$ is {\em{proper}}. 
Unfortunately, our proof needs the assumption that $\pi$ is {\em{projective}} 
because we use Kodaira's lemma for big 
$\mathbb R$-divisors on (not necessarily normal) 
irreducible varieties (cf.~\cite[Lemma A.10]{fujino-slc}). 
Therefore, Theorem \ref{f-thm1.1} is slightly weaker than the original 
statement (see \cite[Theorem 7.2]{ambro}). 

\begin{rem}\label{f-rem1.2}
Precisely speaking, it is sufficient to assume that $\pi$ is proper and that 
every qlc stratum $C$ of $[X, \omega]$ is projective over $S$ in Theorem \ref{f-thm1.1}. 
It is obvious by the proof of Theorem \ref{f-thm1.1}. 
\end{rem}

\begin{rem}\label{f-rem1.3}
In Theorem \ref{f-thm1.1}, 
if $qL-\omega$ is ample, then it is well-known that 
$\mathcal O_X(mL)$ is $\pi$-generated 
for every $m\gg 0$ (see \cite[Theorem 5.1]{ambro}, 
\cite[Theorem 3.66]{fujino-book}, and \cite[Theorem 6.5.1]{fujino-foundation}). 
For the proof, see \cite[Theorem 3.66]{fujino-book} (see 
also \cite[Section 6.5]{fujino-foundation}). 
\end{rem}

We give a remark on \cite[Theorem 7.2]{ambro}. 

\begin{rem}\label{f-rem1.4}
Although Ambro wrote that 
the proof of \cite[Theorem 7.2]{ambro} 
is parallel to 
\cite[Theorem 5.1]{ambro}, 
it does not seem to be true as stated, 
as there are some technical problems in the inductive 
step of the proof. 
Steps 1, 2, and 4 in the proof of 
\cite[Theorem 5.1]{ambro} 
work without any modifications. 
In Step 3, 
$q'L-\omega'$ is $\pi$-nef but 
$q'L-\omega'=qL-\omega$ 
is not always nef and log big over $S$ with 
respect to $[X, \omega']$, 
where $\omega'=\omega+cD$ and $q'=q+cm$. So, 
we can not directly apply the argument in Step 1 in the proof of \cite[Theorem 5.1]{ambro} 
to 
this new quasi-log pair $[X, \omega']$. 
\end{rem}

As a special case of Theorem \ref{f-thm1.1}, we have: 

\begin{thm}[Basepoint-free theorem of Reid--Fukuda type 
for log canonical pairs]\label{f-thm1.5}
Let $(X, B)$ be a log canonical pair. 
Let $L$ be a $\pi$-nef Cartier divisor on $X$ 
where $\pi:X\to S$ is a projective morphism between schemes. 
Assume that $qL-(K_X+B)$ is nef and log big over $S$ 
with respect to $(X, B)$ for some positive real number $q$. 
Then $\mathcal O_X(mL)$ is $\pi$-generated for every $m\gg 0$. 
\end{thm}

Theorem \ref{f-thm1.5} is nothing but \cite[Theorem 4.4]{fujino-book} 
(see \cite[Corollary 6.9.4]{fujino-foundation}). 
We believe that Theorem \ref{f-thm1.5} holds 
under the weaker assumption that $\pi$ is only {\em{proper}}. 
Note that we do not know the proof of Theorem \ref{f-thm1.5} 
without using the theory of quasi-log schemes. 
The usual basepoint-free theorem for log canonical 
pairs, that is, Theorem \ref{f-thm1.5} with the extra assumption that 
$qL-(K_X+B)$ is ample over $S$, 
can be proved without 
using quasi-log structures (see \cite[Theorem 13.1]{fujino}). 
The proof in \cite{fujino} is much simpler than the arguments in 
this paper. 

\begin{rem}\label{f-rem1.6} 
In Theorem \ref{f-thm1.5}, 
if every log canonical center $C$ of $(X, B)$ is 
projective over $S$, then we can 
prove Theorem \ref{f-thm1.5} under the weaker assumption 
that $\pi:X\to S$ is only {\em{proper}}. 
It is because we can apply Theorem \ref{f-thm1.5} to 
the non-klt locus $\Nklt(X, B)$ of $(X, B)$. 
So, we may assume that 
$\mathcal O_X(mL)$ is $\pi$-generated 
on a non-empty Zariski open subset containing 
$\Nklt(X, B)$. 
In this case, we can prove Theorem \ref{f-thm1.5} by 
applying the usual X-method to $L$ on $(X, B)$. 
We note that $C$ is projective over 
$S$ when $\dim C\leq 1$.  

The reader can find a different proof of Theorem \ref{f-thm1.5} 
in \cite{fukuda-lc} 
when $(X, B)$ is a log canonical {\em{surface}}, 
where Fukuda used the log minimal model program with scaling for divisorial 
log terminal surfaces. 
\end{rem}

More generally, we have: 

\begin{thm}\label{f-thm1.7} 
Let $X$ be a normal variety, let $B$ be an effective $\mathbb R$-divisor 
on $X$ such that $K_X+B$ is $\mathbb R$-Cartier, and 
let $\pi:X\to S$ be a projective morphism between schemes. 
Let $L$ be a $\pi$-nef Cartier divisor on $X$ such that 
$qL-(K_X+B)$ is nef and log big over $S$ with respect 
to $(X, B)$ for some positive real number $q$. 
Assume that $\mathcal O_{\mathrm{Nlc}(X, B)}(mL)$ is 
$\pi$-generated for every $m\gg 0$. 
Note that $\mathrm{Nlc}(X, B)$ denotes the non-lc locus of $(X, B)$ 
and is defined by the non-lc ideal sheaf $\mathcal J_{\mathrm{NLC}}(X, B)$ of 
$(X, B)$. Then $\mathcal O_X(mL)$ is $\pi$-generated for 
every $m\gg 0$. 
\end{thm}

For the details of $\mathcal J_{\mathrm{NLC}}(X, B)$, 
see \cite{fujino-nonlc} and \cite[\S7.~Non-lc ideal sheaves]{fujino}. 
Theorem \ref{f-thm1.7} is new and is 
a generalization of \cite[Theorem 13.1]{fujino} 
and \cite[Theorem 9.1]{fujino-surface}. 

In this paper, we use the following convention. 

\begin{notation}\label{f-notation1.8} 
The expression \lq $\ldots$ for every $m\gg 0$\rq \ means that 
\lq there exists a positive integer $m_0$ such that $\ldots$ for every $m\geq m_0$.\rq
\end{notation}

We summarize the contents of this paper. 
In Section \ref{f-sec2}, we recall some basic definitions. 
In Section \ref{f-sec3}, we recall the basic definitions and 
properties of quasi-log schemes. Then 
we introduce various new operations for quasi-log structures 
(see Lemmas \ref{f-lem3.12}, \ref{f-lem3.14}, \ref{f-lem3.15}, and so on). 
Section \ref{f-sec4} is devoted to the proof of 
the main theorem:~Theorem \ref{f-thm1.1}. 

\begin{ack}
The author was partially supported by Gran-in-Aid 
for Young Scientists (A) 24684002 from JSPS. 
He thanks Kento Fujita for comments. 
He also thanks the referee for useful comments. 
\end{ack} 

We will work over $\mathbb C$, the complex number field, throughout this 
paper. 
For the standard notation of the log minimal model program, see, 
for example, \cite{fujino} and \cite{fujino-foundation}. 
For the basic definitions and properties 
of the theory of quasi-log schemes, 
see \cite{fujino-pull} (see also \cite{fujino-foundation}). 
Note that \cite{fujino-foundation} 
is a completely revised and expanded version of 
the author's unpublished manuscript \cite{fujino-book}. 
For a gentle introduction to the theory of quasi-log schemes 
(varieties), we recommend the reader to see \cite{fujino-intro}. 
In this paper, a {\em{scheme}} means 
a separated scheme of finite type over $\Spec \mathbb C$.  
A {\em{variety}} means a reduced scheme. 

\section{Preliminaries}\label{f-sec2}

In this section, let us recall some basic definitions. 

\begin{say}[Operations for $\mathbb R$-divisors]\label{f-say2.1}
Let $D$ be an $\mathbb R$-divisor 
on an equidimensional variety $X$, that is, 
$D$ is a finite formal $\mathbb R$-linear combination 
$$D=\sum _i d_i D_i$$ of irreducible 
reduced subschemes $D_i$ of codimension one. 
We define the {\em{round-up}} 
$\lceil D\rceil =\sum _i \lceil d_i \rceil D_i$ (resp.~{\em{round-down}} 
$\lfloor D\rfloor =\sum _i \lfloor d_i \rfloor D_i$), where for 
every real number $x$, $\lceil x\rceil$ (resp.~$\lfloor x\rfloor$) is the integer 
defined by $x\leq \lceil x\rceil <x+1$ 
(resp.~$x-1<\lfloor x\rfloor \leq x$). The 
{\em{fractional part}} $\{D\}$ of $D$ denotes $D-\lfloor D\rfloor$. We put 
$$D^{<1}=\sum _{d_i<1}d_i D_i,  \quad 
D^{>1}=\sum _{d_i>1}d_i D_i, \quad \text{and}\quad D^{=1}=\sum _{d_i=1}D_i.$$ 
We call $D$ a {\em{boundary}} (resp.~{\em{subboundary}}) $\mathbb R$-divisor if 
$0\leq d_i\leq 1$ (resp.~$d_i\leq 1$) for every $i$. 
\end{say}

\begin{say}[Singularities of pairs]\label{f-say2.2}
Let $X$ be a normal variety and let $\Delta$ be an 
$\mathbb R$-divisor on $X$ 
such that $K_X+\Delta$ is $\mathbb R$-Cartier. 
Let $f:Y\to X$ be 
a resolution such that $\Exc(f)\cup f^{-1}_*\Delta$, 
where $\Exc (f)$ is the exceptional locus of $f$ 
and $f^{-1}_*\Delta$ is 
the strict transform of $\Delta$ on $Y$,  
has a simple normal crossing support. We can 
write 
\begin{equation}\label{f-eq2.1}
K_Y=f^*(K_X+\Delta)+\sum _i a_i E_i. 
\end{equation}
We say that $(X, \Delta)$ 
is {\em{sub log canonical}} ({\em{sub lc}}, for short) if $a_i\geq -1$ for every $i$. 
We usually write $a_i= a(E_i, X, \Delta)$
and call it the {\em{discrepancy coefficient}} of 
$E_i$ with respect to $(X, \Delta)$. 
If $(X, \Delta)$ is sub log canonical and $\Delta$ is effective, then 
$(X, \Delta)$ is called {\em{log canonical}} ({\em{lc}}, for short). 

It is well-known that there is the largest Zariski open subset $U$ 
of $X$ such that  
$(U, \Delta|_U)$ is sub log canonical. 
If there exist a resolution $f:Y\to X$ and a divisor $E$ on $Y$ such 
that $a(E, X, \Delta)=-1$ and $f(E)\cap U\ne \emptyset$, then $f(E)$ is called a 
{\em{log canonical center}} (an {\em{lc center}}, for short) with respect to $(X, \Delta)$. 
A closed subset $C$ of $X$ is called a {\em{log canonical stratum}} 
(an {\em{lc stratum}}, for short) of $(X, \Delta)$ if and only if 
$C$ is a log canonical center of $(X, \Delta)$ or $C$ is 
an irreducible component of $X$. 

From now on, we assume that 
$\Delta$ is effective. In the above formula 
\eqref{f-eq2.1}, we put $\Delta_Y=-\sum _i a_i E_i$. 
Then 
$$
\mathcal J(X, \Delta)=f_*\mathcal O_{Y}(-\lfloor \Delta_Y\rfloor)
$$ 
is a well-defined ideal sheaf on $X$ and is known as the {\em{multiplier 
ideal sheaf}} associated to the pair $(X, \Delta)$. 
The closed subscheme $\Nklt(X, \Delta)$ defined by $\mathcal J(X, \Delta)$ 
is called the {\em{non-klt locus}} of $(X, \Delta)$. 
We put 
$$
\mathcal J_{\mathrm{NLC}}(X, \Delta)=f_*\mathcal O_{Y}(-\lfloor \Delta_Y\rfloor +
\Delta_Y^{=1})
$$ 
and call it the {\em{non-lc ideal sheaf}} associated to the pair $(X, \Delta)$. 
The closed subscheme $\mathrm{Nlc}(X, \Delta)$ is defined by 
$\mathcal J_{\mathrm{NLC}}(X, \Delta)$ and 
is called {\em{the non-lc locus}} of $(X, \Delta)$. 
\end{say}

The notion of {\em{nef and log big divisors}} 
was first introduced in \cite[10.4]{shokurov} by 
Miles Reid. 
For the details of big $\mathbb R$-divisors 
on non-normal irreducible varieties, see \cite[Appendix A]{fujino-slc}. 

\begin{say}[Nef and log big divisors]\label{f-say2.3}
Let $X$ be a normal variety, let $\Delta$ be an effective $\mathbb R$-divisor 
on $X$ such that $K_X+\Delta$ is $\mathbb R$-Cartier, 
and let $\pi:X\to S$ be a proper morphism between schemes. 
Let $L$ be a Cartier divisor on $X$. 
We say that $L$ is {\em{nef and log big}} over $S$ with respect to 
$(X, \Delta)$ if $L$ is nef over $S$ and $L|_C$ is big over 
$S$ for every lc stratum $C$ of $(X, \Delta)$. 
\end{say}

We close this section with: 

\begin{notation}\label{f-notation2.4}
A pair $[X, \omega]$ consists of 
a scheme $X$ and an $\mathbb R$-Cartier 
divisor (or $\mathbb R$-line bundle) on $X$. 
\end{notation}

\section{On quasi-log structures}\label{f-sec3}
In this section, we recall some definitions and basic properties of 
quasi-log schemes and prove some useful lemmas. 
We prove various new lemmas to make the theory of quasi-log schemes 
more flexible and more useful. 
For a quick introduction to the theory of quasi-log schemes (varieties), 
we recommend the reader to see \cite{fujino-intro}.  

Let us quickly recall the definitions of {\em{globally embedded simple 
normal crossing pairs}} and {\em{quasi-log schemes}} for 
the reader's convenience. 
For the details, see, for example, \cite[Section 3]{fujino-pull} 
and \cite[Chapter 5 and Chapter 6]{fujino-foundation}. 

\begin{defn}[Globally embedded simple normal crossing 
pairs]\label{f-def3.1} 
Let $Y$ be a simple normal crossing divisor 
on a smooth 
variety $M$ and let $D$ be an $\mathbb R$-divisor 
on $M$ such that 
$\Supp (D+Y)$ is a simple normal crossing divisor on $M$ and that 
$D$ and $Y$ have no common irreducible components. 
We put $B_Y=D|_Y$ and consider the pair $(Y, B_Y)$. 
We call $(Y, B_Y)$ a {\em{globally embedded simple normal 
crossing pair}} and $M$ the {\em{ambient space}} of $(Y, B_Y)$. 
A {\em{stratum}} of $(Y, B_Y)$ is 
the $\nu$-image of a log canonical stratum of $(Y^\nu, \Theta)$ 
where $\nu:Y^\nu\to Y$ is the normalization and $K_{Y^\nu}+\Theta
=\nu^*(K_Y+B_Y)$. 
\end{defn}

The following lemma is obvious but very important. 

\begin{lem}\label{f-lem3.2} 
Let $Y$ be a smooth irreducible variety and let $B_Y$ be 
an $\mathbb R$-divisor on $Y$ such that 
$\Supp B_Y$ is a simple normal crossing divisor. 
Then $(Y, B_Y)$ is a globally embedded simple normal crossing 
pair. 
\end{lem}

\begin{proof}
We put $M=Y\times \mathbb C$, $D=B_Y\times \mathbb C$, and 
$Y=Y\times 
\{0\}$. 
Then $D$ and $Y$ are divisors on $M$ such that 
$D|_Y=B_Y$. This means that $(Y, B_Y)$ is a globally 
embedded simple normal crossing pair. 
\end{proof}

In this paper, we adopt the following definition of 
quasi-log schemes. 
Although it looks slightly different from Ambro's 
original definition in \cite{ambro}, it is equivalent to 
\cite[Definition 4.1]{ambro}. 

\begin{defn}[Quasi-log schemes]\label{f-def3.3}
A {\em{quasi-log scheme}} is a scheme $X$ endowed with an 
$\mathbb R$-Cartier divisor 
(or $\mathbb R$-line bundle) 
$\omega$ on $X$, a proper closed subscheme 
$X_{-\infty}\subset X$, and a finite collection $\{C\}$ of reduced 
and irreducible subschemes of $X$ such that there is a 
proper morphism $f:(Y, B_Y)\to X$ from a globally 
embedded simple 
normal crossing pair satisfying the following properties: 
\begin{itemize}
\item[(1)] $f^*\omega\sim_{\mathbb R}K_Y+B_Y$. 
\item[(2)] The natural map 
$\mathcal O_X
\to f_*\mathcal O_Y(\lceil -(B_Y^{<1})\rceil)$ 
induces an isomorphism 
$$
\mathcal I_{X_{-\infty}}\overset{\simeq}{\longrightarrow} f_*\mathcal O_Y(\lceil 
-(B_Y^{<1})\rceil-\lfloor B_Y^{>1}\rfloor),  
$$ 
where $\mathcal I_{X_{-\infty}}$ is the defining ideal sheaf of 
$X_{-\infty}$. 
\item[(3)] The collection of subvarieties $\{C\}$ coincides with the image 
of $(Y, B_Y)$-strata that are not included in $X_{-\infty}$. 
\end{itemize}
We simply write $[X, \omega]$ to denote 
the above data 
$$
\bigl(X, \omega, f:(Y, B_Y)\to X\bigr)
$$ 
if there is no risk of confusion. 
Note that a quasi-log scheme $X$ is the union of $\{C\}$ and $X_{-\infty}$. 
We also note that $\omega$ is called the {\em{quasi-log canonical class}} 
of $[X, \omega]$, which is defined up to $\mathbb R$-linear equivalence.  
We sometimes simply say that 
$[X, \omega]$ is a {\em{quasi-log pair}}. 
The subvarieties $C$ 
are called the {\em{qlc strata}} of $[X, \omega]$, 
$X_{-\infty}$ is called the {\em{non-qlc locus}} 
of $[X, \omega]$, and $f:(Y, B_Y)\to X$ is 
called a {\em{quasi-log resolution}} 
of $[X, \omega]$. 
We sometimes use $\Nqlc(X, \omega)$ to denote 
$X_{-\infty}$. 
\end{defn}

For the details of the various equivalent definitions of quasi-log schemes, 
see \cite[Sections 3, 4, and 8]{fujino-pull}. 

\begin{rem}\label{f-rem3.4}
A {\em{qlc stratum}} of $[X, \omega]$ was originally 
called a {\em{qlc center}} of $[X, \omega]$ in the literature. 
We change the terminology (see Definition \ref{f-def3.5} below). 
\end{rem}

Our definition of {\em{qlc centers}} is different from 
Ambro's original one in \cite{ambro}. 

\begin{defn}[Qlc centers]\label{f-def3.5}
A closed subvariety $C$ of $X$ is called a {\em{qlc center}} 
of $[X, \omega]$ if $C$ is a qlc stratum of $[X, \omega]$ which is not 
an irreducible component of $X$. 
\end{defn}

\begin{defn}[Qlc pairs]\label{f-def3.6}Let $[X, \omega]$ be a quasi-log scheme. 
Assume that $X_{-\infty}=\emptyset$. 
Then we sometimes simply say that $[X, \omega]$ is 
a {\em{qlc pair}} or 
$[X, \omega]$ is a quasi-log scheme with only {\em{quasi-log canonical 
singularities}}. 
\end{defn}

We need the notion of {\em{nef and log big divisors}} 
on quasi-log schemes for Theorem \ref{f-thm1.1}. 

\begin{defn}[Nef and log big divisors 
for quasi-log schemes]\label{f-def3.7} 
Let $L$ be an $\mathbb R$-Cartier divisor (or 
$\mathbb R$-line bundle) on a quasi-log pair $[X, \omega]$ and 
let $\pi:X\to S$ be a proper morphism between schemes. 
Then $L$ is {\em{nef and log big over $S$ with 
respect to $[X, \omega]$}} if $L$ is 
$\pi$-nef and $L|_C$ is $\pi$-big for every 
qlc stratum $C$ of $[X, \omega]$. 
\end{defn}

The following theorem is a key result for the theory of quasi-log schemes. 
It follows from the Koll\'ar-type torsion-free and vanishing theorem 
for simple normal crossing varieties. 
For the details, see \cite[Chapter 2]{fujino-book}, \cite{fujino-vanishing}, 
\cite{fujino-inj}, and \cite{fujino-foundation}.  

\begin{thm}[{see \cite[Theorems 4.4 and 7.3]
{ambro}, \cite[Theorem 3.39]{fujino-book}, and 
\cite[Theorem 6.3.4]{fujino-foundation}}]\label{f-thm3.8} 
Let $[X, \omega]$ be a quasi-log scheme and let $X'$ be the union of 
$X_{-\infty}$ with a {\em{(}}possibly empty{\em{)}} union of some 
qlc strata of $[X, \omega]$. Then 
we have the following properties. 
\begin{itemize}
\item[(i)] Assume that $X'\ne X_{-\infty}$. Then 
$X'$ is a quasi-log scheme with $\omega'=\omega|_{X'}$ and 
$X'_{-\infty}=X_{-\infty}$. Moreover, the qlc strata of $[X', \omega']$ are 
exactly the qlc strata of $[X, \omega]$ that are included in $X'$. 
\item[(ii)] Assume that $\pi:X\to S$ is a proper morphism between schemes. 
Let $L$ be a Cartier divisor on $X$ such that 
$L-\omega$ is nef and log big over $S$ with respect to $[X, \omega]$. 
Then $R^i\pi_*(\mathcal I_{X'}\otimes \mathcal O_X(L))=0$ for every $i>0$, 
where $\mathcal I_{X'}$ is the defining ideal sheaf of $X'$ on $X$. 
\end{itemize}
\end{thm}

We give a proof of Theorem \ref{f-thm3.8} for the reader's convenience 
because the theory of quasi-log schemes is not popular yet. 

\begin{proof}
By taking some blow-ups of the ambient space $M$ of $(Y, B_Y)$, we 
may assume that the union of all strata of $(Y, B_Y)$ mapped 
to $X'$, which 
is denoted by $Y'$, is a union of irreducible components 
of $Y$ (see \cite[Proposition 4.1]{fujino-pull}). 
We put $K_{Y'}+B_{Y'}=(K_Y+B_Y)|_{Y'}$ and 
$Y''=Y-Y'$. We will prove that $f:(Y', B_{Y'})\to X'$ 
gives the desired quasi-log structure on $[X', \omega']$.  
By construction, we have 
$f^*\omega'\sim _{\mathbb R}K_{Y'}+B_{Y'}$ on $Y'$. 
We put $A=\lceil -(B^{<1}_Y)\rceil$ and 
$N=\lfloor B^{>1}_Y\rfloor$. We consider the following 
short exact sequence 
$$
0\to \mathcal O_{Y''}(-Y')\to \mathcal O_Y\to \mathcal O_{Y'}\to 0. 
$$
By applying $\otimes \mathcal O_Y(A-N)$, 
we have 
$$
0\to \mathcal O_{Y''}(A-N-Y')\to \mathcal O_Y(A-N)\to \mathcal O_{Y'}
(A-N)\to 0. 
$$ 
By applying $f_*$, we 
obtain 
\begin{align*}
0&\to f_*\mathcal O_{Y''}(A-N-Y')\to f_*\mathcal O_Y(A-N)\to f_*\mathcal O_{Y'}
(A-N)
\\ 
&\to R^1f_*\mathcal O_{Y''}(A-N-Y')\to \cdots. 
\end{align*}
By \cite[Theorem 1.1]{fujino-vanishing} and \cite[Theorem 2.39]{fujino-book} 
(see also \cite[Theorem 5.6.3]{fujino-foundation}), 
no associated prime of 
$R^1f_*\mathcal O_{Y''}(A-N-Y')$ is contained in $X'=f(Y')$. 
We note 
that 
\begin{align*}
(A-N-Y')|_{Y''}-(K_{Y''}+\{B_{Y''}\}+B^{=1}_{Y''}-Y'|_{Y''})
&=-(K_{Y''}+B_{Y''})\\ 
&\sim _{\mathbb R}-(f^*\omega)|_{Y''}, 
\end{align*}
where $K_{Y''}+B_{Y''}=(K_Y+B_Y)|_{Y''}$. 
Therefore, the connecting homomorphism 
$\delta: f_*\mathcal O_{Y'}(A-N)\to 
R^1f_*\mathcal O_{Y''}(A-N-Y')$ is zero. 
Thus we obtain the following short exact sequence 
$$
0\to f_*\mathcal O_{Y''}(A-N-Y')\to \mathcal I_{X_{-\infty}}\to f_*\mathcal O_{Y'}
(A-N)\to 0.  
$$  
We put $\mathcal I_{X'}=f_*\mathcal O_{Y''}(A-N-Y')$. 
Then $\mathcal I_{X'}$ defines 
a scheme structure on $X'$. 
We put $\mathcal I_{X'_{-\infty}}=\mathcal I_{X_{-\infty}}/\mathcal I_{X'}$. 
Then $\mathcal I_{X'_{-\infty}}\simeq f_*\mathcal O_{Y'}(A-N)$ 
by the above exact sequence. 
By the following big commutative diagram: 
$$
\xymatrix{
0
\ar[r] 
&
f_*\mathcal O_{Y''}(A-N-Y') 
\ar[d]\ar[r]
&
f_*\mathcal O_Y(A-N)
\ar[d]\ar[r] 
&
f_*\mathcal O_{Y'}(A-N) 
\ar[r]\ar[d] 
& 0 \\ 
0
\ar[r] 
&
f_*\mathcal O_{Y''}(A-Y') 
\ar[r]
&
f_*\mathcal O_Y(A)
\ar[r] 
&
f_*\mathcal O_{Y'}(A) 
\\
0
\ar[r] 
&
\mathcal I_{X'}
\ar[u]\ar[r]
&
\mathcal O_{X}
\ar[u]\ar[r] 
&
\mathcal O_{X'}
\ar[r]\ar[u] 
& 0, 
}
$$
we can see that $\mathcal O_{X'}\to f_*\mathcal O_{Y'}(
\lceil -(B^{<1}_{Y'})\rceil)$ induces 
an isomorphism $$\mathcal I_{X'_{-\infty}}\overset{\simeq}{\longrightarrow} 
f_*\mathcal O_{Y'}(\lceil -(B^{<1}_{Y'})\rceil-
\lfloor B^{>1}_{Y'}\rfloor).$$  
Therefore, $[X', \omega']$ is a quasi-log pair 
such that $X'_{-\infty}=X_{-\infty}$. By construction, the property on 
qlc strata is obvious. 
So, we obtain the desired quasi-log structure of $[X', \omega']$ 
in (i). 

Let $f:(Y, B_Y)\to X$ be a quasi-log resolution as in the 
proof of (i). 
If $X'=X_{-\infty}$ in the above proof of (i), then 
we can easily see that 
$$
f_*\mathcal O_{Y''}(A-N-Y')\simeq f_*\mathcal O_{Y''}(A-N)\simeq \mathcal I_{X_{-\infty}}. 
$$
Therefore, we always have that $\mathcal I_{X'}\simeq 
f_*\mathcal O_{Y''}(A-N-Y')$. 
Note that $$f^*(L-\omega)\sim _{\mathbb R}
f^*L-(K_{Y''}+B_{Y''})$$ on $Y''$, 
where $K_{Y''}+B_{Y''}=(K_Y+B_Y)|_{Y''}$. 
We also note that 
\begin{align*}
&f^*L-(K_{Y''}+B_{Y''})
\\&=
(f^*L+A-N-Y')|_{Y''}
-(K_{Y''}+\{B_{Y''}\}+B^{=1}_{Y''}-Y'|_{Y''}) 
\end{align*} 
and that no stratum of $(Y'', \{B_{Y''}\}+B^{=1}_{Y''}-Y'|_{Y''})$ is mapped 
to $X_{-\infty}$. Then, 
by \cite[Theorem 3.38]{fujino-book} (see also \cite[Theorem 5.7.3]{fujino-foundation}), 
we have 
$$
R^i\pi_*(f_*\mathcal O_{Y''}(f^*L+A-N-Y'))
=R^i\pi_*(\mathcal I_{X'}\otimes \mathcal O_X(L))=0$$ for 
every $i>0$. Thus, we obtain the desired vanishing theorem in (ii). 
\end{proof}

We usually call Theorem \ref{f-thm3.8} (i) {\em{adjunction}} 
for quasi-log schemes. 

Let us recall the following well-known lemma for the reader's convenience 
(see \cite[Proposition 4.7]{ambro}, 
\cite[Proposition 3.44]{fujino-book}, and 
\cite[Lemma 6.3.5]{fujino-foundation}). 

\begin{lem}\label{f-lem3.9}
Let $[X, \omega]$ be a quasi-log scheme with $X_{-\infty}=\emptyset$. 
Assume that 
$X$ is the unique qlc stratum of $[X, \omega]$. 
Then $X$ is normal. 
\end{lem}

The following proof is different from Ambro's 
original one (see \cite[Proposition 4.7]{ambro}). 

\begin{proof}
Let $f:(Y, B_Y)\to X$ be a quasi-log resolution. 
Since $X_{-\infty}=\emptyset$, 
we have $f_*\mathcal O_Y(\lceil -(B_Y^{<1})\rceil)
\simeq \mathcal O_X$. 
This implies that $f_*\mathcal O_Y\simeq \mathcal O_X$. 
Let $\nu:X^{\nu}\to X$ be the normalization. 
By assumption, $X$ is irreducible and every stratum of 
$(Y, B_Y)$ is mapped onto $X$. 
Thus the indeterminacy locus of 
$\nu^{-1}\circ f:Y\dashrightarrow X^\nu$ contains no strata of $(Y, B_Y)$. 
By modifying $(Y, B_Y)$ suitably by \cite[Proposition 4.1]{fujino-pull}, 
we may assume that $f:Y\to X$ factors through 
$X^\nu$. 
$$
\xymatrix{
Y\ar[d]_{\overline f}\ar[dr]^{f}&\\
X^\nu\ar[r]_{\nu}&X
}
$$
Note that the composition 
$$
\mathcal O_X\to \nu_*\mathcal O_{X^\nu}\to \nu_*\overline f_*\mathcal O_Y
=f_*\mathcal O_Y\simeq \mathcal O_X
$$ 
is an isomorphism. 
This implies that $\mathcal O_X\simeq \nu_*\mathcal O_{X^\nu}$. 
Therefore, $X$ is normal. 
\end{proof}

We introduce $\Nqklt(X, \omega)$, which is 
a generalization of the notion of non-klt loci (see \ref{f-say2.2}). 

\begin{notation}\label{f-notation3.10}
Let $[X, \omega]$ be a quasi-log scheme. 
The union of $X_{-\infty}$ with all qlc centers of $[X, \omega]$ is denoted 
by $\Nqklt(X, \omega)$. 
The scheme structure of $\Nqklt(X, \omega)$ is defined in Theorem 
\ref{f-thm3.8}. If $\Nqklt(X, \omega)\ne X_{-\infty}$, then 
$$[\Nqklt(X, \omega), \omega|_{\Nqklt(X, \omega)}]$$ is a quasi-log scheme 
by Theorem \ref{f-thm3.8}. 
Note that $\Nqklt(X, \omega)$ is denoted by ${\mathrm{LCS}}(X)$ 
and is called the {\em{LCS locus}} 
of a quasi-log scheme $[X, \omega]$ in \cite[Definition 4.6]{ambro}. 
\end{notation}

Theorem \ref{f-thm3.11} is also a key result for the theory of 
quasi-log schemes. 

\begin{thm}[{see \cite[Proposition 4.8]{ambro}, 
\cite[Theorem 3.45]{fujino}, and \cite[Theorem 6.3.7]{fujino-foundation}}]
\label{f-thm3.11} 
Assume that $[X, \omega]$ is a quasi-log scheme with $X_{-\infty}=\emptyset$. 
Then we have the 
following properties. 
\begin{itemize}
\item[(i)] The intersection of two qlc strata 
is a union of qlc strata. 
\item[(ii)] For any closed point $x\in X$, 
the set of all qlc strata passing through 
$x$ has a unique minimal element $C_x$. 
Moreover, $C_x$ is normal at $x$. 
\end{itemize}
\end{thm}

\begin{proof}
Let $C_1$ and $C_2$ be two qlc strata of $[X,\omega]$. 
We fix $P\in C_1\cap C_2$. 
It is enough to find a qlc stratum $C$ such that 
$P\in C\subset C_1\cap C_2$. The union 
$X'=C_1\cup C_2$ with 
$\omega'=\omega|_{X'}$ is a qlc pair having 
two 
irreducible components. 
Hence, it is not normal at $P$. 
By Lemma \ref{f-lem3.9}, 
$P\in \Nqklt (X', \omega')$. Therefore, 
there exists a qlc stratum $C$ 
such that $P\in C\subset X'$. 
We may assume that $C\subset C_1$ with $\dim C<\dim C_1$. 
If $C\subset C_2$, 
then we are done. 
Otherwise, we repeat the argument with $C_1=C$ and 
reach the conclusion in a finite number of 
steps. So, we finish the proof of (i). 
The uniqueness of the minimal qlc stratum follows 
from (i) and the normality of the minimal stratum 
follows from Lemma \ref{f-lem3.9}. 
Thus, we have (ii). 
\end{proof}

The following lemma is very useful for some applications. 
By Lemma \ref{f-lem3.12}, we can throw away the redundant components of 
$Y$ from the quasi-log resolution $f:(Y, B_Y)\to X$. 

\begin{lem}\label{f-lem3.12}
Let $\bigl(X, \omega, f:(Y, B_Y)\to X\bigr)$ be a quasi-log scheme as in Definition \ref{f-def3.3}. 
Then we can construct a new quasi-log resolution 
$f':(Y', B_{Y'})\to X$ such that 
\begin{itemize}
\item[(i)] $f':(Y', B_{Y'})\to X$ gives the same quasi-log structure as one 
given by $f:(Y, B_Y)\to X$, and 
\item[(ii)] every 
irreducible component of $Y'$ is mapped to $\overline{X\setminus 
X_{-\infty}}$, the closure of $X\setminus X_{-\infty}$ in $X$, 
by $f'$. 
\end{itemize} 
\end{lem}
\begin{proof}
Let $M$ be the ambient space of $(Y, B_Y)$. 
By taking some blow-ups of $M$, we may assume that 
the union of all strata of $(Y, B_Y)$ that are not 
mapped 
to $\overline{X\setminus X_{-\infty}}$, which is denoted by $Y''$, 
is a union of some irreducible components of $Y$ (see 
\cite[Proposition 4.1]{fujino-pull}). 
We put $Y'=Y-Y''$ and $K_{Y''}+B_{Y''}=(K_Y+B_Y)|_{Y''}$. 
We may further assume that the union of 
all strata of $(Y, B_Y)$ mapped to 
$\overline {X\setminus X_{-\infty}}
\cap X_{-\infty}$ is a union of some irreducible components of $Y$ by 
\cite[Proposition 4.1]{fujino-pull}. 
We consider the short exact sequence 
$$
0\to \mathcal O_{Y''}(-Y')\to \mathcal O_Y\to \mathcal O_{Y'}\to 0. 
$$ 
We put $A=\lceil -(B_Y^{<1})\rceil$ and $N=\lfloor B_Y^{>1}\rfloor$. By 
applying $\otimes \mathcal O_Y(A-N)$, we have 
$$
0\to \mathcal O_{Y''}(A-N-Y')\to \mathcal O_Y(A-N)\to \mathcal O_{Y'}
(A-N)\to 0. 
$$
By taking $f_*$, we obtain 
\begin{align*}
0&\to f_* \mathcal O_{Y''}(A-N-Y')\to f_*\mathcal O_Y(A-N)\to f_* \mathcal O_{Y'}
(A-N)\\
&\to R^1f_*\mathcal O_{Y''}(A-N-Y')\to \cdots.
\end{align*}
By \cite[Theorem 1.1]{fujino-vanishing} and \cite[Theorem 2.39]{fujino-book} 
(see also \cite[Theorem 5.6.3]{fujino-foundation}), 
no associated prime of $R^1f_*\mathcal O_{Y''}(A-N-Y')$ is contained 
in $f(Y')\cap X_{-\infty}$. 
Note that 
\begin{align*}
(A-N-Y')|_{Y''}-(K_{Y''}+\{B_{Y''}\}+B_{Y''}^{=1}
-Y'|_{Y''})
&=-(K_{Y''}+B_{Y''})
\\&\sim _{\mathbb R} -(f^*\omega)|_{Y''}. 
\end{align*} 
Therefore, the connecting homomorphism 
$$\delta:f_*\mathcal O_{Y'}(A-N)\to R^1f_*\mathcal O_{Y''}(A-N-Y')$$ is zero. 
This implies that 
$$
0\to f_*\mathcal O_{Y''}(A-N-Y')\to \mathcal I_{X_{-\infty}}\to f_*\mathcal O_{Y'}
(A-N)\to 0 
$$ 
is exact. 
The ideal sheaf $\mathcal J=f_*\mathcal O_{Y''}
(A-N-Y')$ is zero 
when it is restricted to $X_{-\infty}$ because 
$\mathcal J\subset \mathcal I_{X_{-\infty}}$. 
On the other hand, $\mathcal J$ is zero on 
$X\setminus X_{-\infty}$ because 
$f(Y'')\subset X_{-\infty}$. 
Therefore, we obtain $\mathcal J=0$. 
Thus we have $\mathcal I_{X_{-\infty}}=f_*\mathcal O_{Y'}(A-N)$. 
So $f'=f|_{Y'}:(Y', B_{Y'})\to X$, where 
$K_{Y'}+B_{Y'}=(K_Y+B_Y)|_{Y'}$, 
gives the same quasi-log structure as one 
given by $f:(Y, B_Y)\to X$ with the property (ii). 
\end{proof}

Lemma \ref{f-lem3.13} is obvious. We will sometimes use it implicitly 
in the theory of quasi-log schemes. 

\begin{lem}\label{f-lem3.13}
Let $[X, \omega]$ be a quasi-log scheme. 
Assume that $X=V\cup X_{-\infty}$ and 
$V\cap X_{-\infty}=\emptyset$. 
Then $[V, \omega|_V]$ is a quasi-log scheme with only quasi-log 
canonical singularities. 
\end{lem}

By using Lemma \ref{f-lem3.12}, we obtain Lemma \ref{f-lem3.14}. 
Roughly speaking, by Lemma \ref{f-lem3.14}, 
we can through away the irreducible components of 
$X$ contained in $X_{-\infty}$ from the quasi-log pair $[X, \omega]$. 

\begin{lem}\label{f-lem3.14}
Let $[X, \omega]$ be a quasi-log scheme. 
We consider $X^\dag=\overline {X\setminus X_{-\infty}}$, 
the closure of $X\setminus X_{-\infty}$ in $X$, with the 
reduced scheme structure. 
Then $[X^\dag, \omega^\dag]$, where $\omega^\dag=\omega|_{X^\dag}$, 
has a natural quasi-log structure induced by $[X, \omega]$. 
This means that 
\begin{itemize}
\item[(i)] $C$ is a qlc stratum of $[X, \omega]$ if and only if 
$C$ is a qlc stratum of $[X^\dag, \omega^\dag]$, and 
\item[(ii)] $\mathcal I_{\Nqlc(X, \omega)}=\mathcal I_{\Nqlc (X^\dag, \omega^\dag)}$. 
\end{itemize}
\end{lem}

\begin{proof}
Let $\mathcal I_{X^\dag}$ be the defining ideal sheaf of $X^\dag$ on $X$. 
Let $f':(Y', B_{Y'})\to X$ be the quasi-log resolution constructed in the proof of 
Lemma \ref{f-lem3.12}. 
Note that 
\begin{align*}
\mathcal I_{X_{-\infty}}&\simeq f'_*\mathcal O_{Y'}(A-N)\\
&\simeq f'_*\mathcal O_{Y'}(-N)
\end{align*} 
and that 
$$
f'(N)=X_{-\infty}\cap f'(Y')=X_{-\infty}\cap X^\dag
$$ 
set theoretically, where 
$A=\lceil -(B_{Y'}^{<1})\rceil$ and $N=\lfloor B_{Y'}^{>1}\rfloor$ 
(see \cite[Remark 3.8]{fujino-pull}). 
Therefore, we obtain 
$$\mathcal I_{X^\dag}\cap \mathcal I_{X_{-\infty}}
=\mathcal I_{X^\dag}\cap f'_*\mathcal O_{Y'}(-N')=\{0\}.$$ 
Thus we can construct the following big commutative diagram. 
$$
\xymatrix{&&0\ar[d]&0\ar[d]&
\\&&\mathcal I_{X_{-\infty}}\ar@{=}[r]\ar[d]& \mathcal I_{X^\dag_{-\infty}}\ar[d]&\\
0\ar[r]&\mathcal I_{X^\dag}\ar[r]\ar@{=}[d]&\mathcal O_X\ar[r]\ar[d]
&\mathcal O_{X^\dag}\ar[r]\ar[d]&0\\
0\ar[r]&\mathcal I_{X^\dag}\ar[r]&\mathcal O_{X_{-\infty}}\ar[r]\ar[d]
&\mathcal O_{X^\dag_
{-\infty}}\ar[d]\ar[r]&0\\
&&0&0&
}
$$
By construction, $f'$ factors through $X^\dag$. 
We put $g:(Y', B_{Y'})\to X^\dag$. 
Then $g:(Y', B_{Y'})\to X^\dag$ gives the desired quasi-log structure 
on $[X^\dag, \omega^\dag]$. 
\end{proof}

We need Lemma \ref{f-lem3.15} in the proof of Theorem \ref{f-thm1.1}. 

\begin{lem}\label{f-lem3.15}
Let $[X, \omega]$ be a quasi-log scheme and let $E$ be an effective $\mathbb R$-Cartier 
$\mathbb R$-divisor on $X$. 
We put 
$$\widetilde \omega=\omega+\varepsilon E$$ with $0<\varepsilon \ll 1$. 
Then $[X, \widetilde \omega]$ has a natural quasi-log structure 
with the following properties. 
\begin{itemize}
\item[(i)] Let $\{C_i\}_{i\in I}$ be the set of 
qlc strata of $[X, \omega]$ contained in $\Supp E$. 
We put $$X^{\star}=(\underset{i\in I}{\cup}C_i)\cup \Nqlc(X, \omega)$$ 
as in Theorem \ref{f-thm3.8}. 
Then $\Nqlc(X, \widetilde \omega)$ coincides with $X^{\star}$ scheme theoretically. 
\item[(ii)] $C$ is a qlc stratum of 
$[X, \widetilde \omega]$ if and only if $C$ is a qlc stratum of $[X, \omega]$ with $C
\not\subset \Supp E$.  
\end{itemize}
\end{lem}

\begin{proof}
Let $f:(Y, B_Y)\to X$ be a quasi-log resolution as in Definition \ref{f-def3.3}. 
By \cite[Proposition 4.1]{fujino-pull}, 
the union of 
all strata of $(Y, B_Y)$ mapped to $X^{\star}$, which is 
denoted by $Y''$, is a union of 
some irreducible components of $Y$. 
We put $Y'=Y-Y''$ and $K_{Y'}+B_{Y'}=(K_Y+B_Y)|_{Y'}$. 
We may further assume that $(Y', B_{Y'}+f^*E)$ is 
a globally embedded simple normal crossing pair by \cite[Proposition 4.1]{fujino-pull}. 
We consider $f:(Y', B_{Y'}+\varepsilon f^*E)\to X$ with 
$0<\varepsilon \ll 1$. 
We put $A=\lceil -(B_Y^{<1})\rceil$ and $N=\lfloor B_{Y}^{>1}\rfloor$. 
Then $X^{\star}$ is defined by the ideal 
sheaf $f_*\mathcal O_{Y'}(A-N-Y'')$ (see the 
proof of Theorem \ref{f-thm3.8}). 
Note that 
\begin{align*}
(A-N-Y'')|_{Y'}
&=-\lfloor B_{Y'}+\varepsilon f^*E\rfloor+(B_{Y'}+\varepsilon f^*E)^{=1}
\\&=\lceil -(B_{Y'}+\varepsilon f^*E)^{<1}\rceil 
-\lfloor (B_{Y'}+\varepsilon f^*E)^{>1}\rfloor. 
\end{align*}
Therefore, if we define $\Nqlc (X, \widetilde \omega)$ by the ideal sheaf 
$$
f_*\mathcal O_{Y'}(\lceil -(B_{Y'}+\varepsilon f^*E)^{<1}\rceil 
-\lfloor (B_{Y'}+\varepsilon f^*E)^{>1}\rfloor)=f_*\mathcal O_{Y'}
(A-N-Y''), 
$$ 
then $f:(Y', B_{Y'}+\varepsilon f^*E)\to X$ 
gives the desired quasi-log structure on $[X, \widetilde \omega]$. 
\end{proof}

The following lemma is 
a slight generalization of \cite[Lemma 3.71]{fujino-book}, which 
played a crucial role in the proof of the rationality theorem for 
quasi-log schemes (see \cite[Theorem 3.68]{fujino-book} 
and \cite[Theorem 6.6.1]{fujino-foundation}). 

\begin{lem}[{see \cite[Lemma 3.71]{fujino-book} and 
\cite[Lemma 6.3.9]{fujino-foundation}}]\label{f-lem3.16} 
Let $[X, \omega]$ be a quasi-log scheme with 
$X_{-\infty}=\emptyset$ and 
let $x\in X$ be a closed point. 
Let $D_1, D_2, \cdots, D_k$ be effective Cartier divisors on $X$ such that 
$x\in \Supp D_i$ for every $i$. 
Let $f:(Y, B_Y)\to X$ be a quasi-log resolution. 
Assume that the normalization 
of $(Y, B_Y+\sum _{i=1}^k f^*D_i)$ is sub log canonical. 
This means that $(Y^\nu, \Xi)$ is sub log canonical, where 
$\nu:Y^\nu\to Y$ is the normalization and 
$K_{Y^\nu}+\Xi=\nu^*(K_Y+B_Y+\sum _{i=1}^kf^*D_i)$. 
Note that it requires that no irreducible component of $Y$ is mapped 
into $\cup _{i=1}^k \Supp D_i$. 
Then $k\leq \dim _xX$. 
More precisely, $k\leq \dim _xC_x$, where 
$C_x$ is the minimal qlc stratum of $[X, \omega]$ passing 
through $x$.  
\end{lem}

\begin{proof}We prove this lemma by induction on the dimension. 
\begin{step}\label{f-step1}
By \cite[Proposition 4.1]{fujino-pull}, 
we may assume that $(Y, B_Y+\sum _{i=1}^kf^*D_i)$ is 
a globally embedded simple normal crossing pair. 
Let $i_0$ be any positive integer with 
$1\leq i_0\leq k$. 
Note that $f_*\mathcal O_{Y}(\lceil -(B_{Y}^{<1})\rceil)\simeq \mathcal O_X$. 
Therefore, for any irreducible component $T$ of $\Supp D_{i_0}$, 
there is a stratum $S$ of $(Y, B_Y+f^*D_{i_0})$ mapped onto 
$T$. 
Note that $f:(Y, B_Y+\sum _{i=1}^kf^*D_i)\to X$ gives 
a natural quasi-log structure on $[X, \omega+\sum _{i=1}^kD_i]$ with 
only quasi-log canonical singularities. We also note that 
$\Supp D_i$ and $\Supp D_j$ have no common irreducible components for 
$i\ne j$ by the condition $f_*\mathcal O_Y(\lceil -(B^{<1}_Y)\rceil)\simeq 
\mathcal O_X$. 
\end{step}
\begin{step}\label{f-step2}
In this step, we assume that $\dim _xX=1$. 
If $x$ is a qlc stratum of $[X, \omega]$, then we have $k=0$. 
Therefore, we may assume that $x$ is not a qlc stratum of $[X, \omega]$. 
By shrinking $X$ around $x$, we may assume that 
every stratum of $(Y, B_Y)$ is mapped onto $X$. 
Then $X$ is irreducible and normal (see Lemma \ref{f-lem3.9}), 
and $f:Y\to X$ is flat. 
In this case, $f_*\mathcal O_Y(\lceil -(B_Y^{<1})\rceil)\simeq \mathcal O_X$ 
implies $k\leq 1=\dim _xX$. 
\end{step}
\begin{step}\label{f-step3}
We assume that $\dim _x X\geq 2$. 
If $x$ is a qlc stratum of $[X, \omega]$, then $k=0$. 
So we may assume that 
$x$ is not a qlc stratum of $[X, \omega]$. 
Let $C$ be the minimal qlc stratum of $[X, \omega]$ passing through 
$x$. 
By shrinking $X$ around $x$, 
we may assume that $C$ is normal (see Theorem \ref{f-thm3.11}). 
By \cite[Proposition 4.1]{fujino-pull}, 
we may assume that the union of all strata of $(Y, B_Y)$ mapped to 
$C$, which is denoted by $Y'$, is a union of some irreducible components of $Y$. 
Then $f:(Y', B_{Y'})\to C$ gives a natural quasi-log structure induced by 
the original quasi-log structure $f:(Y, B_Y)\to X$ 
(see Theorem \ref{f-thm3.8}). 
Therefore, by 
induction on the dimension, 
we have $k\leq \dim _xC\leq \dim _xX$ when $\dim _xC<\dim _xX$. 
Thus we may assume that $X$ is the unique qlc stratum of 
$[X, \omega]$. 
Note that $f:(Y, B_Y+f^*D_1)\to X$ gives a natural quasi-log 
structure on $[X, \omega+D_1]$ with only quasi-log canonical singularities. 
Let $X'$ be the union of qlc strata of $[X, \omega+D_1]$ contained 
in $\Supp D_1$. 
Then $[X', (\omega+D_1)|_{X'}]$ is a qlc pair 
with $\dim _xX'<\dim _x X$ (see Step \ref{f-step1}). 
Note that $[X', (\omega+D_1)|_{X'}]$ with 
$D_2|_{X'}, \cdots, D_k|_{X'}$ satisfies 
the condition similar to the original one for $[X, \omega]$ with 
$D_1, \cdots, D_k$ (see Step \ref{f-step1}). Therefore, 
$k-1\leq \dim _xX'<\dim _xX$. 
This implies $k\leq \dim _xX$. 
\end{step} 
Anyway, we obtained the desired inequality $k\leq \dim _xC_x$, where $C_x$ is the 
minimal qlc stratum of $[X, \omega]$ passing through $x$. 
\end{proof}

\section{Proof of the main theorem}\label{f-sec4}

In this section, we prove Theorem \ref{f-thm1.1}. 

\begin{proof}[Proof of Theorem \ref{f-thm1.1}]
We divide the proof into several steps. 
\setcounter{step}{0}
\begin{step}\label{f-4step1}
If $\dim X\setminus X_{-\infty}=0$, then Theorem \ref{f-thm1.1} obviously holds true. 
From now on, we assume that Theorem \ref{f-thm1.1} holds for 
any quasi-log scheme $Z$ with 
$\dim Z\setminus Z_{-\infty}<
\dim X\setminus X_{-\infty}$. 
\end{step}
\begin{step}\label{f-4step2}
We take a qlc stratum $C$ of $[X, \omega]$. 
We put $X'=C\cup X_{-\infty}$. 
Then $X'$ has a natural quasi-log 
structure induced by $[X, \omega]$ 
(see Theorem \ref{f-thm3.8}). 
By the vanishing theorem (see Theorem \ref{f-thm3.8}), 
we have 
$R^1\pi_*
(\mathcal I_{X'}\otimes \mathcal O_X(mL))=0$ for every $m\geq q$. 
Therefore, we obtain that 
$\pi_*\mathcal O_X(mL)\to \pi_*\mathcal O_{X'}
(mL)$ is surjective for 
every $m\geq q$. 
Thus, we may assume that $\overline {X\setminus X_{-\infty}}$ 
is irreducible for the proof of 
Theorem \ref{f-thm1.1} 
by the following commutative 
diagram. 
$$
\xymatrix{
\pi^*\pi_*\mathcal O_X(mL)\ar[r]\ar[d]&\pi^*\pi_*\mathcal O_{X'}(mL)\ar[r]\ar[d]&0\\
\mathcal O_X(mL)\ar[r]&\mathcal O_{X'}(mL)\ar[r]&0}
$$
\end{step} 
\begin{step}\label{f-4step3}
Let $f:(Y, B_Y)\to X$ be a quasi-log resolution. 
By Lemma \ref{f-lem3.12}, we may assume that 
every irreducible component of $Y$ is mapped to $\overline {X\setminus X_{-\infty}}$. 
We may further assume that $S$ is affine. 
\end{step}
\begin{step}\label{f-4step4}
In this step, we assume that $X$ is the disjoint union of $X_{-\infty}$ and 
a qlc stratum $C$ of $[X, \omega]$. We further assume that 
$C$ is the unique qlc stratum of $[X, \omega]$. 
In this case, we may assume that $X_{-\infty}=\emptyset$ 
by Lemma \ref{f-lem3.12}. 
By Lemma \ref{f-lem3.9}, we have that 
$X$ is normal. 
By Kodaira's lemma, we can write 
$qL-\omega\sim _{\mathbb R} A+E$ on $X$ such that 
$A$ is a $\pi$-ample $\mathbb Q$-divisor on $X$ and 
$E$ is an effective $\mathbb R$-Cartier $\mathbb R$-divisor on $X$. 
We put $\widetilde \omega=\omega+\varepsilon E$ with 
$0<\varepsilon \ll 1$. 
Then $[X, \widetilde\omega]$ is a quasi-log scheme with $\Nqlc(X, \widetilde \omega)=
\emptyset$ (see Lemma \ref{f-lem3.15}). 
Note that 
$$
qL-\widetilde \omega\sim _{\mathbb R}(1-\varepsilon)(qL-\omega)+\varepsilon A
$$ 
is $\pi$-ample. 
Therefore, by the basepoint-free theorem for quasi-log schemes 
(see \cite[Theorem 5.1]{ambro}, \cite[Theorem 3.66]{fujino-book}, 
Remark \ref{f-rem1.3}, and \cite[Theorem 6.5.1]{fujino-foundation}), 
we obtain that 
$\mathcal O_X(mL)$ is $\pi$-generated for every $m\gg 0$. 
\end{step}

\begin{step}\label{f-4step5} 
From now on, by Step \ref{f-4step4}, 
we may assume that there is a qlc center 
$C'$ of $[X, \omega]$ 
or assume that $C\cap X_{-\infty}\ne \emptyset$, 
where $X=C\cup X_{-\infty}$. 
We put 
$$
X'=(\underset{i\in I}{\cup}C_i)\cup X_{-\infty} 
$$ 
as in Theorem \ref{f-thm3.8}, 
where $\{C_i\}_{i\in I}$ is the set of all qlc centers of $[X, \omega]$, 
equivalently, $X'=\Nqklt(X, \omega)$. 
Then, by induction on the dimension or 
the assumption on $\mathcal O_{X_{-\infty}}(mL)$, $\mathcal O_{X'}(mL)$ is 
$\pi$-generated for every $m\gg 0$. 
By the same arguments as in Step \ref{f-4step2}, that is, 
the surjectivity of the restriction map 
$\pi_*\mathcal O_X(mL)\to \pi_*\mathcal O_{X'}(mL)$ 
for every $m\geq q$, 
$\mathcal O_X(mL)$ is $\pi$-generated in a neighborhood $U$ of $X'$ 
for every large and positive integer $m$. 
Note that $C\cap U\ne \emptyset$. 
In particular, for every prime number $p$ and every large 
positive integer $l$, 
$\mathcal O_{X}(p^lL)$ is 
$\pi$-generated in the above neighborhood $U$ of $X'=\Nqklt(X, \omega)$.  
\end{step}

\begin{step}\label{f-4step6} 
In this step, we prove the following claim. 
\begin{claim}
If the relative base locus $\Bs_{\pi}|p^lL|$ {\em{(}}with the reduced 
scheme structure{\em{)}} is not empty, then there is a positive integer $a$ such that 
$\Bs_{\pi}|p^{al}L|$ is strictly smaller than 
$\Bs_{\pi}|p^{l}L|$. 
\end{claim}
\begin{proof}[Proof of Claim] 
Note that $\Bs_{\pi}|p^{al}L|\subseteq \Bs_{\pi}|p^lL|$ for every 
positive integer $a$. 
We consider $[X^\dag, \omega^\dag]$ as in Lemma \ref{f-lem3.14}.
Since $(qL-\omega)|_{X^\dag}$ is nef and big over $S$, we can 
write $$qL|_{X^\dag}-\omega^\dag
\sim_{\mathbb R}A+E$$ on $X^\dag$ by Kodaira's lemma 
(cf.~\cite[Lemma A.10]{fujino-slc}), 
where $A$ is a $\pi$-ample 
$\mathbb Q$-divisor on $X^\dag$ and $E$ is an effective $\mathbb R$-Cartier 
$\mathbb R$-divisor 
on $X^\dag$. 
We note that $X^\dag$ is projective 
over $S$ and that $X^\dag$ is not necessarily normal 
(see Example \ref{f-example4.1}). 
By Lemma \ref{f-lem3.15}, we have a new 
quasi-log structure on $[X^\dag, \widetilde \omega]$, 
where $\widetilde \omega=\omega^\dag+\varepsilon E$ 
with $0<\varepsilon \ll 1$, such that 
\begin{equation}\label{f-eq4.1}
\Nqlc(X^\dag, \widetilde {\omega})=(\underset{i\in I}{\cup} C_i)\cup 
\Nqlc(X^\dag, \omega^\dag), 
\end{equation} 
where $\{C_i\}_{i\in I}$ is the set of qlc centers of $[X^\dag, 
\omega^\dag]$ contained 
in $\Supp E$. 

We put $n=\dim X^\dag$. Let 
$D_1, \cdots, D_{n+1}$ be general members of $|p^lL|$. 
Let $f:(Y, B_Y)\to X^\dag$ be a quasi-log resolution of $[X^\dag, \widetilde
\omega]$. 
We consider $f:(Y, B_Y+\sum _{i=1}^{n+1}f^*D_i)\to X^\dag$. 
We put 
$$
c=\underset{t\geq 0}{\sup}\left\{t \left|
\begin{array}{l}  {\text{the normalization of 
$(Y, B_Y+t\sum_{i=1}^{n+1}f^*D_i)$ is}}\\
{\text{sub log canonical over $X^\dag\setminus\Nqlc(X^\dag, \widetilde {\omega})$}} 
\end{array}\right. \right\}. 
$$
Then we have 
$c<1$ by Lemma \ref{f-lem3.16}. We have 
$0<c$ by Step \ref{f-4step5}. Thus, 
$$f:(Y, B_Y+c\sum _{i=1}^{n+1}f^*D_i)\to X^\dag$$ gives a quasi-log 
structure on $[X^\dag, \widetilde {\omega}+c\sum _{i=1}^{n+1}D_i]$. 
Note that $[X^\dag, \widetilde {\omega}+c\sum _{i=1}^{n+1}D_i]$ 
has only quasi-log canonical singularities on $X^\dag\setminus 
\Nqlc(X^\dag, \widetilde \omega)$. 
By construction, there is a qlc center $C_0$ of $[X^\dag, \widetilde \omega+c
\sum _{i=1}^{n+1}D_i]$ contained in $\Bs_{\pi}|p^lL|$. 
We put $\widetilde \omega+c
\sum _{i=1}^{n+1}D_i=\overline \omega$. 
Then 
$$
C_0\cap \Nqlc(X^\dag, \overline \omega)=\emptyset 
$$ 
because 
$$
\Bs_{\pi}|p^lL|\cap \Nqklt (X, \omega)=\emptyset. 
$$ 
Note that $\Nqlc(X^\dag, \overline \omega)=\Nqlc(X^\dag, \widetilde \omega)$ 
by construction. 
We also note that 
\begin{align*}
(q+c(n+1)p^l)L|_{X^\dag}-\overline \omega\sim _{\mathbb R}(1-\varepsilon)
(qL|_{X^\dag}-\omega^\dag)+\varepsilon A
\end{align*} 
is ample over $S$. 
Therefore, 
\begin{equation}\label{f-eq4.2}
\pi_*\mathcal O_{X^\dag}(mL)\to \pi_*\mathcal O_{C_0}
(mL)\oplus \pi_*\mathcal O_{\Nqlc(X^\dag, \overline \omega)}(mL)
\end{equation}
is surjective for every $m\geq q+c(n+1)p^l$. 
Moreover, $\pi_*\mathcal O_{C_0}(mL)$ is $\pi$-generated 
for every $m\gg 0$ by the basepoint-free theorem for quasi-log schemes 
(see \cite[Theorem 5.1]{ambro}, \cite[Theorem 3.66]{fujino-book}, 
Remark \ref{f-rem1.3}, and \cite[Theorem 6.5.1]{fujino-foundation}). 
Note that $[C_0, {\overline \omega}|_{C_0}]$ is 
a quasi-log scheme with only quasi-log canonical singularities 
by Theorem \ref{f-thm3.8} and Lemma \ref{f-lem3.13}. 
Therefore, we can construct a section $s$ 
of $\mathcal O_{X^\dag}(p^{al}L)$ for some positive integer $a$ such that 
$s|_{C_0}$ is not zero and $s$ is zero on $\Nqlc(X^\dag, \overline{\omega})$ 
by \eqref{f-eq4.2}. 
Thus $s$ is zero on $$\Nqlc(X^\dag, \overline{\omega})=\Nqlc 
(X^\dag, \widetilde{\omega})=(\cup _{i\in I} C_i )\cup \Nqlc (X^\dag, 
\omega^\dag)$$ by \eqref{f-eq4.1}. 
In particular, $s$ is zero on $\Nqlc(X^\dag, \omega^\dag)$. 
So, $s$ can be seen as a section of $\mathcal O_X(p^{al}L)$ 
because 
$\mathcal I_{\Nqlc(X^\dag, \omega^\dag)}=\mathcal I_{\Nqlc(X, \omega)}$ 
by construction (see Lemma \ref{f-lem3.14}). 
Therefore, $\Bs_{\pi}|p^{al}L|$ is strictly smaller than $\Bs_{\pi}|p^lL|$. 
We complete the proof of Claim. 
\end{proof}
\end{step}
\begin{step}\label{f-4step7}
By Step \ref{f-4step6} and the noetherian induction, $\mathcal O_{X}(p^lL)$ and 
$\mathcal O_X(p'^{l'}L)$ are both $\pi$-generated for large 
$l$ and $l'$, where $p$ and $p'$ are distinct prime numbers. 
So, there exists a positive integer $m_0$ such that 
$\mathcal O_X(mL)$ is $\pi$-generated for every $m\geq m_0$. 
\end{step}
Thus we obtain the desired basepoint-free theorem. 
\end{proof}

\begin{ex}\label{f-example4.1}
Let $C$ be a nodal curve on a smooth 
surface. Then $[C, K_C]$ is a quasi-log scheme with 
only quasi-log canonical singularities. 
In this case, $C$ is not normal. 
\end{ex}

Finally, we prove Theorem \ref{f-thm1.5} and Theorem \ref{f-thm1.7}. 

\begin{proof}[Proof of Theorem \ref{f-thm1.5} and Theorem \ref{f-thm1.7}]
Let $f:Y\to X$ be a resolution such that 
$K_Y+B_Y=f^*(K_X+B)$ and $\Supp B_Y$ is a simple normal crossing 
divisor on $Y$. 
By Lemma \ref{f-lem3.2}, $(Y, B_Y)$ is a globally embedded simple normal crossing 
pair. 
Then $f:(Y, B_Y)\to X$ defines a quasi-log structure on $[X, K_X+B]$. 
We note that $\mathcal J_{\mathrm{NLC}}(X, B)$ 
coincides with the defining ideal sheaf of $\Nqlc(X, K_X+B)$ and 
that $C$ is a qlc stratum of $[X, K_X+B]$ if and only if $C$ is 
a log canonical stratum of $(X, B)$. 
Therefore, Theorem \ref{f-thm1.7} is a special case of Theorem \ref{f-thm1.1}. 
Moreover, Theorem \ref{f-thm1.5} is a special case of Theorem \ref{f-thm1.7}. 
\end{proof}


\end{document}